\documentclass[english,reqno]{amsart}
\usepackage[T1]{fontenc}
\usepackage[latin9]{inputenc}
\usepackage{babel}
\usepackage{refstyle}
\usepackage{float}
\usepackage{mathrsfs}
\usepackage{mathtools}
\usepackage{enumitem}
\usepackage{amstext}
\usepackage{amsthm}
\usepackage{amssymb}
\usepackage{graphicx}
\usepackage[all]{xy}
\PassOptionsToPackage{normalem}{ulem}
\usepackage{ulem}
\usepackage[unicode=true,pdfusetitle,
 bookmarks=true,bookmarksnumbered=false,bookmarksopen=false,
 breaklinks=false,pdfborder={0 0 1},backref=false,colorlinks=false]
 {hyperref}
\hypersetup{
 colorlinks=true,citecolor=blue,linkcolor=blue,linktocpage=true}

\makeatletter

\AtBeginDocument{\providecommand\remref[1]{\ref{rem:#1}}}
\AtBeginDocument{\providecommand\thmref[1]{\ref{thm:#1}}}
\AtBeginDocument{\providecommand\lemref[1]{\ref{lem:#1}}}
\AtBeginDocument{\providecommand\exaref[1]{\ref{exa:#1}}}
\AtBeginDocument{\providecommand\subsecref[1]{\ref{subsec:#1}}}
\AtBeginDocument{\providecommand\figref[1]{\ref{fig:#1}}}
\AtBeginDocument{\providecommand\defref[1]{\ref{def:#1}}}
\RS@ifundefined{subsecref}
  {\newref{subsec}{name = \RSsectxt}}
  {}
\RS@ifundefined{thmref}
  {\def\RSthmtxt{theorem~}\newref{thm}{name = \RSthmtxt}}
  {}
\RS@ifundefined{lemref}
  {\def\RSlemtxt{lemma~}\newref{lem}{name = \RSlemtxt}}
  {}

\numberwithin{equation}{section}
\numberwithin{figure}{section}
\numberwithin{table}{section}
\theoremstyle{plain}
\newtheorem{thm}{\protect\theoremname}[section]
  \theoremstyle{definition}
  \newtheorem{defn}[thm]{\protect\definitionname}
  \theoremstyle{remark}
  \newtheorem{rem}[thm]{\protect\remarkname}
  \theoremstyle{plain}
  \newtheorem{lem}[thm]{\protect\lemmaname}
  \theoremstyle{plain}
  \newtheorem*{question*}{\protect\questionname}
  \theoremstyle{definition}
  \newtheorem{example}[thm]{\protect\examplename}
  \theoremstyle{plain}
  \newtheorem{cor}[thm]{\protect\corollaryname}
  \theoremstyle{plain}
  \newtheorem{question}[thm]{\protect\questionname}
  \theoremstyle{remark}
  \newtheorem*{acknowledgement*}{\protect\acknowledgementname}

\providecommand{\MR}[1]{}

\usepackage{bbm}
\allowdisplaybreaks
\usepackage{needspace}

\setlist[enumerate]{itemsep=5pt,topsep=3pt}
\setlist[enumerate,1]{label=\textup{(}\roman*\textup{)},ref=\roman*}
\setlist[enumerate,2]{label=\textup{(}\alph*\textup{)},ref=\theenumi \alph*}

\newref{lem}{refcmd={Lemma \ref{#1}}}
\newref{thm}{refcmd={Theorem \ref{#1}}}
\newref{cor}{refcmd={Corollary \ref{#1}}}
\newref{sec}{refcmd={Section \ref{#1}}}
\newref{subsec}{refcmd={Section \ref{#1}}}
\newref{chap}{refcmd={Chapter \ref{#1}}}
\newref{prop}{refcmd={Proposition \ref{#1}}}
\newref{exa}{refcmd={Example \ref{#1}}}
\newref{tab}{refcmd={Table \ref{#1}}}
\newref{rem}{refcmd={Remark \ref{#1}}}
\newref{def}{refcmd={Definition \ref{#1}}}
\newref{fig}{refcmd={Figure \ref{#1}}}
\newref{claim}{refcmd={Claim \ref{#1}}}

\AtBeginDocument{
  
}

\makeatother

  \providecommand{\acknowledgementname}{Acknowledgement}
  \providecommand{\corollaryname}{Corollary}
  \providecommand{\definitionname}{Definition}
  \providecommand{\examplename}{Example}
  \providecommand{\lemmaname}{Lemma}
  \providecommand{\questionname}{Question}
  \providecommand{\remarkname}{Remark}
\providecommand{\theoremname}{Theorem}

\begin{document}

\title[Sampling with positive definite kernels]{Sampling with positive definite kernels and an associated dichotomy}

\author{Palle Jorgensen}

\address{(Palle E.T. Jorgensen) Department of Mathematics, The University
of Iowa, Iowa City, IA 52242-1419, U.S.A. }

\email{palle-jorgensen@uiowa.edu}

\urladdr{http://www.math.uiowa.edu/\textasciitilde{}jorgen/}

\author{Feng Tian}

\address{(Feng Tian) Department of Mathematics, Hampton University, Hampton,
VA 23668, U.S.A.}

\email{feng.tian@hamptonu.edu}
\begin{abstract}
We study classes of reproducing kernels $K$ on general domains; these
are kernels which arise commonly in machine learning models; models
based on certain families of reproducing kernel Hilbert spaces. They
are the positive definite kernels $K$ with the property that there
are countable discrete sample-subsets $S$; i.e., proper subsets $S$
having the property that every function in $\mathscr{H}\left(K\right)$
admits an $S$-sample representation. We give a characterizations
of kernels which admit such non-trivial countable discrete sample-sets.
A number of applications and concrete kernels are given in the second
half of the paper.
\end{abstract}



\maketitle
\tableofcontents{}

\section{Introduction}

In the theory of non-uniform sampling, one studies Hilbert spaces
consisting of signals, understood in a very general sense. One then
develops analytic tools and algorithms, allowing one to draw inference
for an ``entire'' (or global) signal from partial information obtained
from carefully chosen distributions of sample points. While the better
known and classical sampling algorithms (Shannon and others) are based
on interpolation, modern theories go beyond this. An early motivation
is the work of Henry Landau, see e.g., \cite{MR0129065,MR0140733,MR0206615,MR0222554,doi:10.1137/0144089,MR799420}.
In this setting, it is possible to make precise the notion of ``average
sampling rates'' in general configurations of sample points. (See
also \cite{MR2587581,MR2868037}.)

When a positive definite kernel $K$ is given, we denote by $\mathscr{H}\left(K\right)$
the associated reproducing kernel Hilbert space (RKHS). In the present
paper we study classes of reproducing kernels $K$ on general domains,
such kernels arise commonly in machine learning models based on reproducing
kernel Hilbert space (see e.g., \cite{MR3450534}) with the property
that there are non-trivial restrictions to countable discrete sample
subsets $S$ such that every function in $\mathscr{H}\left(K\right)$
has an $S$-sample representation. In this general framework, we study
properties of positive definite kernels $K$ with respect to sampling
from ``small\textquotedblright{} subsets, and applying to all functions
in the associated Hilbert space $\mathscr{H}\left(K\right)$. We are
motivated by concrete kernels which are used in a number of applications,
for example, on one extreme, the Shannon kernel for band-limited functions,
which admits many sampling realizations; and on the other, the covariance
kernel of Brownian motion which has no non-trivial countable discrete
sample subsets.

We offer an operator theoretic condition which explains, in a general
context, this dichotomy. Our study continues our earlier papers on
reproducing kernels and their restrictions to countable discrete subsets;
see e.g., \cite{zbMATH06664785,MR3402823,MR3450534,2015arXiv150202549J},
and also \cite{MR2810909,MR2591839,MR2488871,MR2228737,MR2186447}.

A reproducing kernel Hilbert space (RKHS) is a Hilbert space $\mathscr{H}$
of functions on a prescribed set, say $T$, with the property that
point-evaluation for functions $f\in\mathscr{H}$ is continuous with
respect to the $\mathscr{H}$-norm. They are called kernel spaces,
because, for every $t\in T$, the point-evaluation for functions $f\in\mathscr{H}$,
$f\left(t\right)$ must then be given as a $\mathscr{H}$-inner product
of $f$ and a vector $K_{t}$, in $\mathscr{H}$; called the kernel.

The RKHSs have been studied extensively since the pioneering papers
by Aronszajn \cite{Aro43,Aro48}. They further play an important role
in the theory of partial differential operators (PDO); for example
as Green's functions of second order elliptic PDOs \cite{Nel57,HKL14}.
Other applications include engineering, physics, machine-learning
theory \cite{KH11,MR2488871,CS02}, stochastic processes \cite{AD93,ABDdS93,AD92,AJSV13,MR3251728},
numerical analysis, and more \cite{MR2089140,MR2607639,MR2913695,MR2975345,MR3091062,MR3101840,MR3201917,Shawe-TaylorCristianini200406,SchlkopfSmola200112}. 

An illustration from \emph{neural networks}: An Extreme Learning Machine
(ELM) is a neural network configuration in which a hidden layer of
weights are randomly sampled \cite{RW06}, and the object is then
to determine analytically resulting output layer weights. Hence ELM
may be thought of as an approximation to a network with infinite number
of hidden units.

Given a positive definite kernel $K:T\times T\rightarrow\mathbb{C}$
(or $\mathbb{R}$ for simplification), there are several notions and
approaches to sampling (i.e., an algorithmic reconstruction of suitable
functions from values at a fixed and pre-selected set of sample-points):
\begin{defn}
We say that $K$ has non-trivial sampling property, if there exists
a countable subset $S\subset T$, and $a,b\in\mathbb{R}_{+}$, such
that 
\begin{equation}
a\sum_{s\in S}\left|f\left(s\right)\right|^{2}\leq\left\Vert f\right\Vert _{\mathscr{H}\left(K\right)}^{2}\leq b\sum_{s\in S}\left|f\left(s\right)\right|^{2},\quad\forall f\in\mathscr{H}\left(K\right),\label{eq:sp1}
\end{equation}
where $\mathscr{H}\left(K\right)$ is the reproducing kernel Hilbert
space (RKHS) of $K$, see \cite{Aro43} and \remref{rk} below.

Suppose equality holds in (\ref{eq:sp1}) with $a=b=1$; then we say
that $\left\{ K\left(\cdot,s\right)\right\} _{s\in S}$ is a \emph{Parseval
frame}. 

It follows that sampling holds in the form
\[
f\left(t\right)=\sum_{s\in S}f\left(s\right)K\left(t,s\right),\quad\forall f\in\mathscr{H}\left(K\right),\:\forall t\in T
\]
if and only if $\left\{ K\left(\cdot,s\right)\right\} _{s\in S}$
is a Parseval frame; see also \thmref{ps}. 
\end{defn}

As is well known, when a vector $f$ in a Hilbert space $\mathscr{H}$
is expanded in an orthonormal basis (ONB) $B$, there is then automatically
an associated Parseval identity. In physical terms, this identity
typically reflects a \emph{stability} feature of a decomposition based
on the chosen ONB $B$. Specifically, Parseval's identity reflects
a conserved quantity for a problem at hand, for example, energy conservation
in quantum mechanics.

The theory of frames begins with the observation that there are useful
vector systems which are in fact not ONBs but for which a Parseval
formula still holds. In fact, in applications it is important to go
beyond ONBs. While this viewpoint originated in signal processing
(in connection with frequency bands, aliasing, and filters), the subject
of frames appears now to be of independent interest in mathematics.
See, e.g., \cite{MR2837145,MR3167899,MR2367342,MR2147063}, and also
\cite{CoDa93,MR2154344,Dutkay_2006}. 

\begin{rem}
\label{rem:rk}To make the discussion self-contained, we add the following
(for the benefit of the readers.)
\begin{enumerate}
\item A given $K:T\times T\rightarrow\mathbb{C}$ is positive definite (p.d.)
if and only if for all $n\in\mathbb{N}$, $\left\{ \xi\right\} _{j=1}^{n}\subset\mathbb{C}$,
and all $\left\{ t_{j}\right\} _{j=1}^{n}\subset T$, we have: 
\[
\sum_{i}\sum_{j}\overline{\xi}_{i}\xi_{j}K\left(t_{i},t_{j}\right)\geq0.
\]
\item \label{enu:rh2}A function $f$ on $T$ is in $\mathscr{H}\left(K\right)$
if and only if there is a constant $C=C\left(f\right)$ such that
for all $n$, $\left(\xi_{j}\right)_{1}^{n}$, $\left(t_{j}\right)_{1}^{n}$,
as above, we have 
\begin{equation}
\left|\sum_{1}^{n}\xi_{j}f\left(t_{j}\right)\right|^{2}\leq C\sum_{i}\sum_{j}\overline{\xi}_{i}\xi_{j}K\left(t_{i},t_{j}\right).\label{eq:rh2}
\end{equation}
\end{enumerate}
It follows from the above that reproducing kernel Hilbert spaces (RKHS)
arise from a given positive definite kernel $K$, a corresponding
pre-Hilbert form; and then a Hilbert-completion. The question arises:
\textquotedblleft What are the functions in the completion?\textquotedblright{}
The \emph{a priori} estimate (\ref{eq:rh2}) in (\ref{enu:rh2}) above
is an answer to the question. We will return to this issue in the
application section 3 below. By contrast, the Hilbert space completions
are subtle; they are classical Hilbert spaces of functions, not always
transparent from the naked kernel $K$ itself. Examples of classical
RKHSs: Hardy spaces or Bergman spaces (for complex domains), Sobolev
spaces and Dirichlet spaces \cite{MR3054607,MR2892621,MR2764237}
(for real domains, or for fractals), band-limited $L^{2}$ functions
(from signal analysis), and Cameron-Martin Hilbert spaces (see \lemref{cm})
from Gaussian processes (in continuous time domain).
\end{rem}

\begin{lem}
\label{lem:fr}Suppose $K$, $T$, $a$, $b$, and $S$ satisfy the
condition in (\ref{eq:sp1}), then there is a positive operator $B$
in $\mathscr{H}\left(K\right)$ with bounded inverse such that 
\[
f\left(\cdot\right)=\sum_{s\in S}\left(Bf\right)\left(s\right)K\left(\cdot,s\right)
\]
is a convergent interpolation formula valid for all $f\in\mathscr{H}\left(K\right)$. 

Equivalently, 
\[
f\left(t\right)=\sum_{s\in S}f\left(s\right)B\left(K\left(\cdot,s\right)\right)\left(t\right),\;\text{for all \ensuremath{t\in T}.}
\]
\end{lem}

\begin{proof}
Define $A:\mathscr{H}\left(K\right)\rightarrow l^{2}\left(S\right)$
by $\left(Af\right)\left(s\right)=f\left(s\right)$, $s\in S$; or
\[
Af:=\left(f\left(s\right)\right)_{s\in S}\in l^{2}\left(S\right).
\]
Then the adjoint operator $A^{*}:l^{2}\left(S\right)\rightarrow\mathscr{H}\left(K\right)$
is given by 
\[
A^{*}\xi=\sum_{s\in S}\xi_{s}K\left(\cdot,s\right),\;\forall\xi\in l^{2}\left(S\right),
\]
and 
\[
A^{*}Af=\sum_{s\in S}f\left(s\right)K\left(\cdot,s\right)
\]
holds in $\mathscr{H}\left(K\right)$, with $\mathscr{H}\left(K\right)$-norm
convergence. This is immediate from (\ref{eq:sp1}).

Now set $B=\left(A^{*}A\right)^{-1}$. Note that 
\[
\left\Vert B\right\Vert _{\mathscr{H}\left(K\right)\rightarrow\mathscr{H}\left(K\right)}\leq a^{-1}
\]
where $a$ is in the lower bound in (\ref{eq:sp1}). 
\end{proof}
\begin{lem}
\label{lem:span}Suppose $K$, $T$, $a$, $b$, and $S$ satisfy
(\ref{eq:sp1}), then the linear span of $\left\{ K\left(\cdot,s\right)\right\} _{s\in S}$
is dense in $\mathscr{H}\left(K\right)$. 
\end{lem}

\begin{proof}
Let $f\in\mathscr{H}\left(K\right)$, then 
\begin{gather*}
f\perp\left\{ K\left(\cdot,s\right)\right\} _{s\in S}\\
\Updownarrow\\
f\left(s\right)=\left\langle K\left(\cdot,s\right),f\right\rangle _{\mathscr{H}\left(K\right)}=0,\;\forall s\in S,
\end{gather*}
by the reproducing property in $\mathscr{H}\left(K\right)$. But by
(\ref{eq:sp1}), $b<\infty$, this implies that $f=0$ in $\mathscr{H}\left(K\right)$.
Hence the family $\left\{ K\left(\cdot,s\right)\right\} _{s\in S}$
has dense span. 
\end{proof}

\section{The dichotomy}

We now turn to dichotomy: (i) Existence of countably discrete sampling
sets vs (ii) non-existence. To help readers appreciate the nature
of the two classes, we begin with two examples, (i) Shannon\textquoteright s
kernel for band-limited functions, \exaref{shan}; and (ii) the covariance
kernel for standard Brownian motion, \thmref{bm}.
\begin{question*}
~
\begin{enumerate}
\item Given a positive definite kernel $K:T\times T\rightarrow\mathbb{R}$,
how to determine when there exist $S\subset T$, and $a,b\in\mathbb{R}_{+}$
such that (\ref{eq:sp1}) holds?
\item Given $K$, $T$ as above, how to determine if there is a countable
discrete subset $S\subset T$ such that 
\begin{equation}
\left\{ K\left(\cdot,s\right)\right\} _{s\in S}\label{eq:sp4}
\end{equation}
 has dense span in $\mathscr{H}\left(K\right)$? 
\end{enumerate}
\end{question*}
\begin{example}
\label{exa:shan}Let $T=\mathbb{R}$, and let $K:\mathbb{R}\times\mathbb{R}\rightarrow\mathbb{R}$
be the Shannon kernel, where 
\begin{align}
K\left(s,t\right) & :=\text{sinc}\,\pi\left(s-t\right)\nonumber \\
 & =\frac{\sin\pi\left(s-t\right)}{\pi\left(s-t\right)},\quad\forall s,t\in\mathbb{R}.\label{eq:sp5}
\end{align}

We may choose $S=\mathbb{Z}$, and then $\left\{ K\left(\cdot,n\right)\right\} _{n\in\mathbb{Z}}$
is even an orthonormal basis (ONB) in $\mathscr{H}\left(K\right)$,
but there are many other examples of countable discrete subsets $S\subset\mathbb{R}$
such that (\ref{eq:sp1}) holds for finite $a,b\in\mathbb{R}_{+}$. 

The RKHS of $K$ in (\ref{eq:sp5}) is the Hilbert space $\subset L^{2}\left(\mathbb{R}\right)$
consisting of all $f\in L^{2}\left(\mathbb{R}\right)$ such that $suppt(\hat{f})\subset\left[-\pi,\pi\right]$,
where ``suppt'' stands for support of the Fourier transform $\hat{f}$.
Note $\mathscr{H}\left(K\right)$ consists of functions on $\mathbb{R}$
which have entire analytic extensions to $\mathbb{C}$; see \cite{MR2039503,MR2040080,MR2975345,MR1976867}.
Using the above observations, we get 
\begin{align*}
f\left(t\right) & =\sum_{n\in\mathbb{Z}}f\left(n\right)K\left(t,n\right)\\
 & =\sum_{n\in\mathbb{Z}}f\left(n\right)\text{sinc}\,\pi\left(t-n\right),\quad\forall t\in\mathbb{R},\:\forall f\in\mathscr{H}\left(K\right).
\end{align*}
\end{example}

\begin{example}
Let $K$ be the covariant kernel of standard Brownian motion, with
$T:=[0,\infty)$, or $T:=[0,1)$. Then 
\begin{equation}
K\left(s,t\right):=s\wedge t=\min\left(s,t\right),\;\forall\left(s,t\right)\in T\times T.\label{eq:sp6}
\end{equation}
\end{example}

\begin{lem}
\label{lem:cm}Let $K$, $T$ be as in (\ref{eq:sp6}). Then $\mathscr{H}\left(K\right)$
consists of functions $f$ on $T$ such that $f$ has distribution
derivative $f'\in L^{2}\left(T,\lambda\right)$, i.e., $L^{2}$ with
respect to Lebesgue measure $\lambda$ on $T$, and 
\begin{equation}
\left\Vert f\right\Vert _{\mathscr{H}\left(K\right)}^{2}=\int_{T}\left|f'\left(x\right)\right|^{2}dx.\label{eq:sp8}
\end{equation}
\end{lem}

\begin{proof}
This is well-known, see e.g., \cite{MR3450534,MR3402823,Hi80}.
\end{proof}
\begin{rem}[see also \subsecref{bm} below]
The significance of (\ref{eq:sp8}) for Brownian motion is as follows: 

Fix $T$, and set $L^{2}\left(T\right)=$ the $L^{2}$-space from
the restriction to $T$ of Lebesgue measure on $\mathbb{R}$. Pick
an ONB $\left\{ \psi_{k}\right\} $ in $L^{2}\left(T\right)$, for
example a Haar-Walsh orthonormal basis in $L^{2}\left(T\right)$.
Let $\left\{ Z_{k}\right\} $ be an i.i.d. (independent identically
distributed) $N\left(0,1\right)$ system, i.e., standard Gaussian
copies. Then 
\begin{equation}
B_{t}\left(\cdot\right)=\sum_{k}\left(\int_{0}^{t}\psi_{k}\left(s\right)ds\right)Z_{k}\left(\cdot\right)\label{eq:sb}
\end{equation}
is a realization of standard Brownian motion on $T$; in particular
we have 
\[
\mathbb{E}\left(B_{s}B_{t}\right)=s\wedge t=K\left(s,t\right),\;\forall\left(s,t\right)\in T\times T.
\]
See \figref{bm}.
\end{rem}

\begin{figure}[H]
\includegraphics[width=0.35\textwidth]{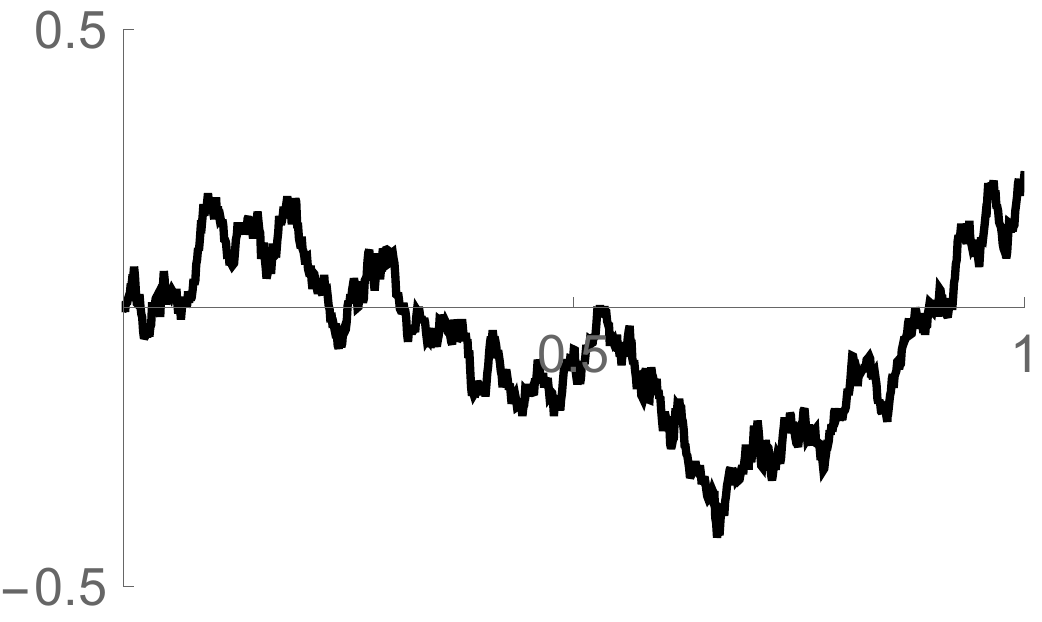}

\caption{\label{fig:bm}Brownian motion; see (\ref{eq:sb}).}

\end{figure}

\begin{thm}
\label{thm:bm}Let $K$, $T$ be as in (\ref{eq:sp6}); then there
is no countable discrete subset $S\subset T$ such that $\left\{ K\left(\cdot,s\right)\right\} _{s\in S}$
is dense in $\mathscr{H}\left(K\right)$. 
\end{thm}

\begin{proof}
Suppose $S=\left\{ x_{n}\right\} $, where 
\begin{equation}
0<x_{1}<x_{2}<\cdots<x_{n}<x_{n+1}<\cdots;\label{eq:sp7}
\end{equation}
then consider the following function
\begin{equation}
\raisebox{-6mm}{\includegraphics[width=0.7\textwidth]{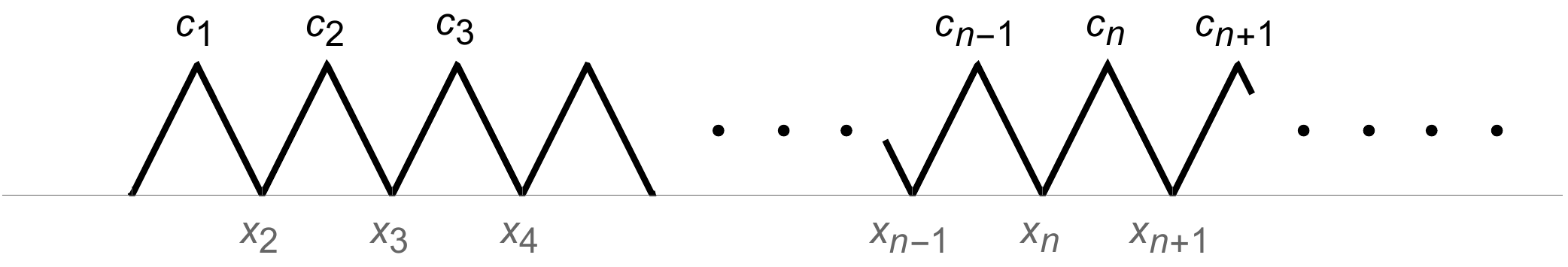}}\label{eq:sp9}
\end{equation}

On the respective intervals $\left[x_{n},x_{n+1}\right]$, the function
$f$ is as follows:
\[
f\left(x\right)=\begin{cases}
c_{n}\left(x-x_{n}\right) & \text{if }x_{n}\leq x\leq\frac{x_{n}+x_{n+1}}{2}\\
c_{n}\left(x_{n+1}-x\right) & \text{if }\frac{x_{n}+x_{n+1}}{2}<x\leq x_{n+1}.
\end{cases}
\]
In particular, $f\left(x_{n}\right)=f\left(x_{n+1}\right)=0$, and
on the midpoints: 
\[
f\left(\frac{x_{n}+x_{n+1}}{2}\right)=c_{n}\frac{x_{n+1}-x_{n}}{2},
\]
see \figref{stooth}.
\begin{figure}[H]
\includegraphics[width=0.35\textwidth]{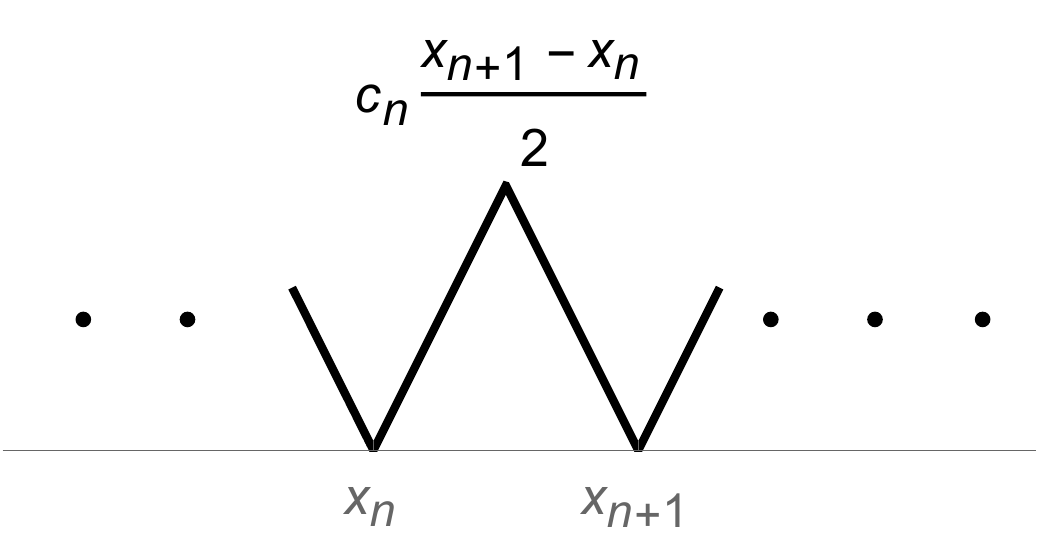}

\caption{\label{fig:stooth}The saw-tooth function.}
\end{figure}

Choose $\left\{ c_{n}\right\} _{n\in\mathbb{N}}$ such that 
\begin{equation}
\sum_{n\in\mathbb{N}}\left|c_{n}\right|^{2}\left(x_{n+1}-x_{n}\right)<\infty.\label{eq:sp11}
\end{equation}
Admissible choices for the slope-values $c_{n}$ include 
\[
c_{n}=\frac{1}{n\sqrt{x_{n+1}-x_{n}}},\;n\in\mathbb{N}.
\]

We will now show that $f\in\mathscr{H}\left(K\right)$. To do this,
use (\ref{eq:sp8}). For the distribution derivative computed from
(\ref{eq:sp9}), we get 
\begin{equation}
\raisebox{-12mm}{\includegraphics[width=0.7\textwidth]{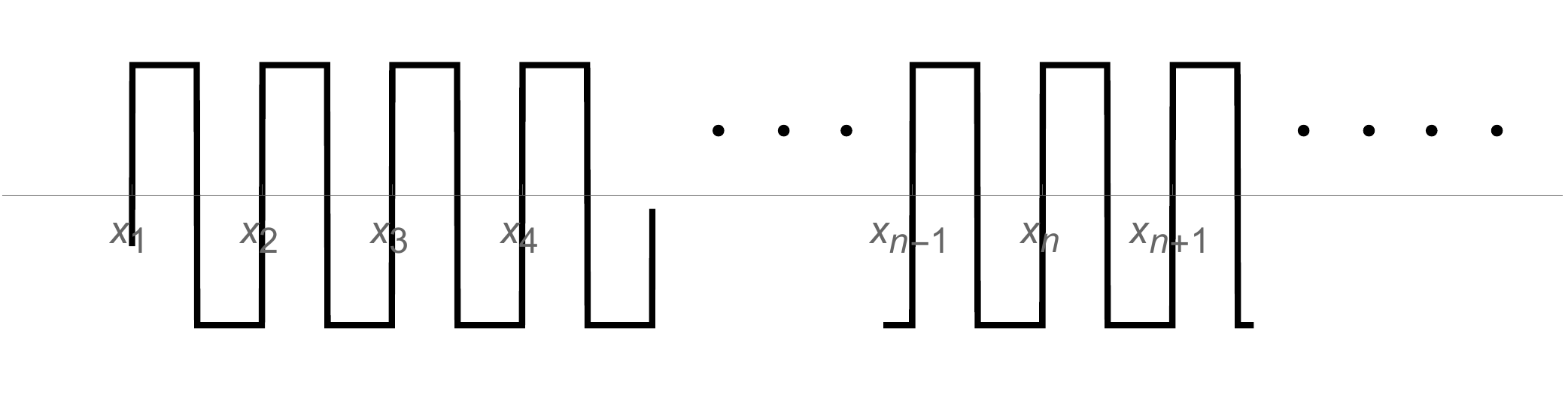}}\label{eq:sp9b}
\end{equation}
\[
\int_{0}^{\infty}\left|f'\left(x\right)\right|^{2}dx=\sum_{n\in\mathbb{N}}\left|c_{n}\right|^{2}\left(x_{n+1}-x_{n}\right)<\infty
\]
which is the desired conclusion, see (\ref{eq:sp9}).
\end{proof}
\begin{cor}
For the kernel $K\left(s,t\right)=s\wedge t$ in (\ref{eq:sp6}),
$T=[0,\infty)$, the following holds:

Given $\left\{ x_{j}\right\} _{j\in\mathbb{N}}\subset\mathbb{R}_{+}$,
$\left\{ y_{j}\right\} _{j\in\mathbb{N}}\subset\mathbb{R}$, then
the interpolation problem 
\begin{equation}
f\left(x_{j}\right)=y_{j},\;f\in\mathscr{H}\left(K\right)\label{eq:ip1}
\end{equation}
is solvable if 
\begin{equation}
\sum_{j\in\mathbb{N}}\left(y_{j+1}-y_{j}\right)^{2}/\left(x_{j+1}-x_{j}\right)<\infty.\label{eq:sp2}
\end{equation}
\end{cor}

\begin{proof}
Let $f$ be the piecewise linear spline (see \figref{ip}) for the
problem (\ref{eq:ip1}), see \figref{ip}; then the $\mathscr{H}\left(K\right)$-norm
is as follows: 
\[
\int_{0}^{\infty}\left|f'\left(x\right)\right|^{2}dx=\sum_{j\in\mathbb{N}}\left(\frac{y_{j+1}-y_{j}}{x_{j+1}-x_{j}}\right)^{2}\left(x_{j+1}-x_{j}\right)<\infty
\]
when (\ref{eq:sp2}) holds. 
\end{proof}
\begin{figure}[H]
\includegraphics[width=0.35\textwidth]{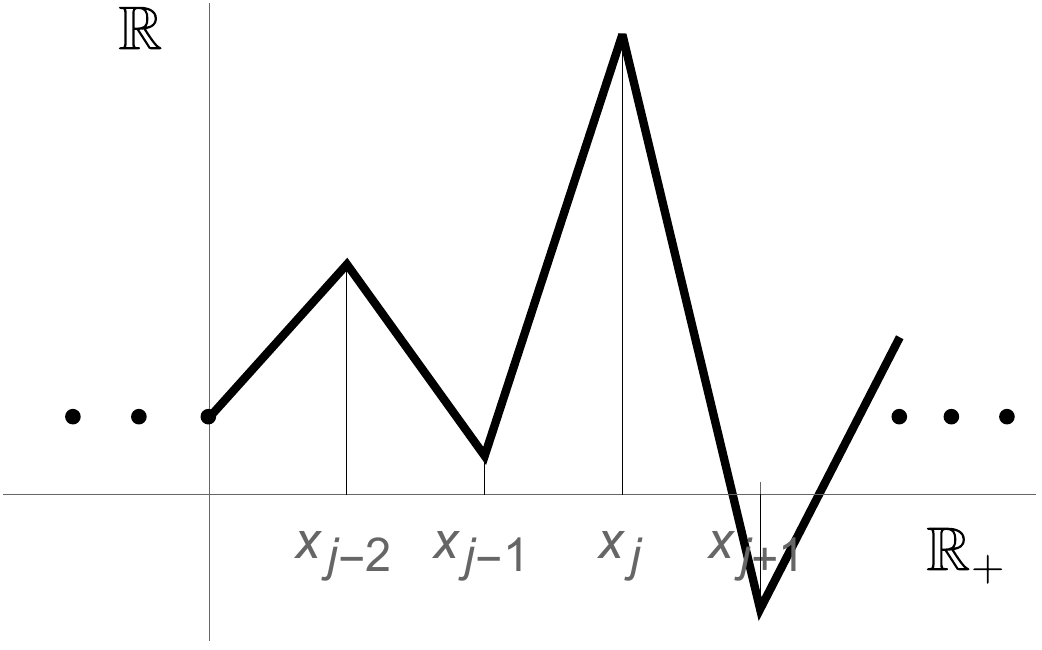}

\caption{\label{fig:ip}Piecewise linear spline.}

\end{figure}
\begin{rem}
Let $K$ be as in (\ref{eq:sp6}), where 
\[
K\left(s,t\right)=s\wedge t,\quad s,t\in[0,\infty).
\]
For all $0\leq x_{j}<x_{j+1}<\infty$, let 
\begin{align*}
f_{j}\left(x\right): & =\frac{2}{x_{j+1}-x_{j}}\left(K\left(x-x_{j},\frac{x_{j+1}-x_{j}}{2}\right)-K\left(x-\frac{x_{j}+x_{j+1}}{2},\frac{x_{j+1}-x_{j}}{2}\right)\right)\\
 & =\raisebox{-5mm}{\includegraphics[width=0.4\textwidth]{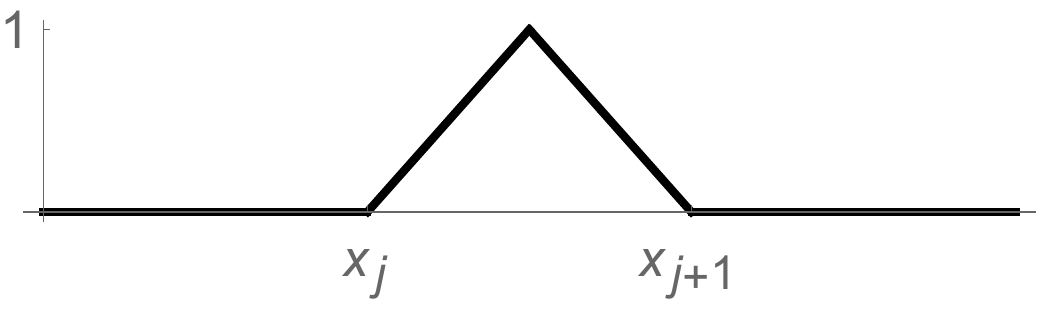}}
\end{align*}
Assuming (\ref{eq:sp11}) holds, then 
\[
f\left(x\right)=\sum_{j}c_{j}f_{j}\left(x\right)\in\mathscr{H}\left(K\right).
\]
\end{rem}

\begin{rem}
Let $K\left(s,t\right)=s\wedge t$, $\left(s,t\right)\in[0,\infty)\times[0,\infty)$,
extend to $\widetilde{K}\left(s,t\right)=\left|s\right|\wedge\left|t\right|$,
$\left(s,t\right)\in\mathbb{R}\times\mathbb{R}$, and $\mathscr{H}(\widetilde{K})=$
all $f$ on $\mathbb{R}$ such that the distribution derivative $f'$
exists on $\mathbb{R}$, and 
\[
\left\Vert f\right\Vert _{\mathscr{H}(\widetilde{K})}^{2}=\int_{\mathbb{R}}\left|f'\left(x\right)\right|^{2}dx.
\]
 
\end{rem}

\begin{thm}
Let $T$ be a set of cardinality $c$ of the continuum, and let $K:T\times T\rightarrow\mathbb{R}$
be a positive definite kernel. Let $S=\left\{ x_{j}\right\} _{j\in\mathbb{N}}$
be a discrete subset of $T$. Suppose there are weights $\left\{ w_{j}\right\} _{j\in\mathbb{N}}$,
$w_{j}\in\mathbb{R}_{+}$, such that 
\begin{equation}
\left(f\left(x_{j}\right)\right)\in l^{2}\left(\mathbb{N},w\right)\label{eq:c1}
\end{equation}
for all $f\in\mathscr{H}\left(K\right)$. Suppose further that there
is a point $t_{0}\in T\backslash S$, a $y_{0}\in\mathbb{R}\backslash\left\{ 0\right\} $,
and $\alpha\in\mathbb{R}_{+}$ such that the infimum 
\begin{equation}
\inf_{f\in\mathscr{H}\left(K\right)}\left\{ \sum\nolimits _{j}w_{j}\left|f\left(x_{j}\right)\right|^{2}+\left|f\left(t_{0}\right)-y_{0}\right|^{2}+\alpha\left\Vert f\right\Vert _{\mathscr{H}\left(K\right)}^{2}\right\} \label{eq:c2}
\end{equation}
is strictly positive. 

Then $S$ is \uline{not} a interpolation set for $\left(K,T\right)$. 
\end{thm}

\begin{proof}
Let $L$ denote the analysis operator defined from condition (\ref{eq:c1})
in the statement of the theorem; see also the beginning in the proof
of \lemref{fr} above, and let $L^{*}$ denote the corresponding adjoint
operator, the synthesis operator. Using now \cite{MR2228737,MR3450534},
we conclude that the function $f$ which minimizes the problem (\ref{eq:c2})
is unique, and in fact
\begin{equation}
f=\left(\alpha I+L^{*}L\right)^{-1}L^{*}\left(\left(y_{j}\right)\cup\left(t_{0}\right)\right).\label{eq:c3}
\end{equation}
So, by the hypothesis in the theorem, we get $f\in\mathscr{H}\left(K\right)\backslash\left\{ 0\right\} $,
and $f\left(x_{j}\right)=0$, for all $j\in\mathbb{N}$. Then it follows
that the closed span of $\left\{ K\left(\cdot,x_{j}\right)\right\} _{j\in\mathbb{N}}$
is not $\mathscr{H}\left(K\right)$; specifically, $0\neq f\in\left\{ K\left(\cdot,x_{j}\right)\right\} _{j\in\mathbb{N}}^{\perp}$.
See also \lemref{span} and \figref{as}.
\end{proof}
\begin{figure}[H]
\[
\xymatrix{l^{2}\left(\left\{ x_{j}\right\} \cup\left\{ t_{0}\right\} ,w\right)\ar@/^{1.5pc}/[r]^{L^{*}} & \mathscr{H}\left(K\right)\ar@/^{1.5pc}/[l]^{L}}
\]

\caption{\label{fig:as}Analysis and synthesis operators.}
\end{figure}
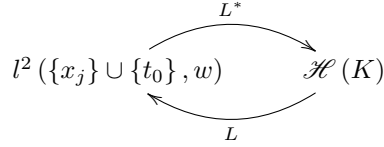
\begin{thm}
\label{thm:ps}Let $K:T\times T\rightarrow\mathbb{R}$ be a positive
definite kernel, and let $S\subset T$ be a countable discrete subset.
The RKHS $\mathscr{H}\left(K\right)$ refers to the pair $\left(K,T\right)$.
For all $s\in S$, set $K_{s}\left(\cdot\right)=K\left(\cdot,s\right)$.
The the following four conditions are equivalent: 
\begin{enumerate}
\item \label{enu:ps1}The family $\left\{ K_{s}\right\} _{s\in S}$ is a
Parseval frame in $\mathscr{H}\left(K\right)$; 
\item \label{enu:ps2}
\[
\left\Vert f\right\Vert _{\mathscr{H}\left(K\right)}^{2}=\sum_{s\in S}\left|f\left(s\right)\right|^{2},\;\forall f\in\mathscr{H}\left(K\right);
\]
\item \label{enu:ps3}
\[
K\left(t,t\right)=\sum_{s\in S}\left|K\left(t,s\right)\right|^{2},\;\forall t\in T;
\]
\item \label{enu:ps4}
\[
f\left(t\right)=\sum_{s\in S}f\left(s\right)K\left(t,s\right),\;\forall f\in\mathscr{H}\left(K\right),\:\forall t\in T,
\]
where the sum converges in the norm of $\mathscr{H}\left(K\right)$. 
\end{enumerate}
\end{thm}

\begin{proof}
(\ref{enu:ps1}) $\Rightarrow$ (\ref{enu:ps2}). Assume (\ref{enu:ps1}),
and note that 
\begin{equation}
\left\langle K_{s},f\right\rangle _{\mathscr{H}\left(K\right)}=f\left(s\right);\label{eq:ps1}
\end{equation}
and (\ref{enu:ps2}) is immediate from the definition of Parseval-frame. 

(\ref{enu:ps2}) $\Rightarrow$ (\ref{enu:ps3}). Assume (\ref{enu:ps2}),
and set $f=K_{t}$. Note that $\left\Vert K_{t}\right\Vert _{\mathscr{H}\left(K\right)}^{2}=K\left(t,t\right)$,
and $\left\langle K_{s},K_{t}\right\rangle _{\mathscr{H}\left(K\right)}=K\left(s,t\right)$. 

(\ref{enu:ps3}) $\Rightarrow$ (\ref{enu:ps4}). It is enough to
prove that 
\begin{equation}
K_{t}=\sum_{s\in S}K\left(t,s\right)K_{s},\;\forall t\in T;\label{eq:ps2}
\end{equation}
then (\ref{enu:ps4}) follows from an application of the reproducing
property of the Hilbert space $\mathscr{H}\left(K\right)$. Now (\ref{eq:ps2})
follows from 
\begin{equation}
\left\Vert K_{t}-\sum\nolimits _{s\in S}K\left(t,s\right)K_{s}\right\Vert _{\mathscr{H}\left(K\right)}^{2}=0.\label{eq:ps3}
\end{equation}
Finally, (\ref{eq:ps3}) follows from (\ref{enu:ps3}) and multiple
application of the kernel property: 
\[
\text{LHS}{}_{\left(\ref{eq:ps3}\right)}=K\left(t,t\right)+\underset{\left(s,s'\right)\in S\times S}{\sum\sum}K\left(t,s\right)K\left(t,s'\right)K\left(s',s\right)-2\sum_{s\in S}\left|K\left(t,s\right)\right|^{2}=0.
\]

(\ref{enu:ps4}) $\Rightarrow$ (\ref{enu:ps1}). It is clear that
(\ref{enu:ps1}) $\Leftrightarrow$ (\ref{enu:ps2}), and that (\ref{enu:ps4})
$\Rightarrow$ (\ref{enu:ps2}).
\end{proof}
\begin{rem}[Stationary kernels]
Suppose $K:\mathbb{R}\times\mathbb{R}\rightarrow\mathbb{R}$ is a
continuous positive definite kernel, and $K\left(s,t\right)=k\left(s-t\right)$,
i.e., stationary. Set $K_{t}\left(\cdot\right):=K\left(\cdot,t\right)=k\left(\cdot-t\right)$.
By Bochner's theorem, 
\[
k\left(t\right)=\int_{\mathbb{R}}e^{itx}d\mu\left(x\right),
\]
where $\mu$ is a finite positive Borel measure on $\mathbb{R}$.
Thus, 
\[
V:K_{t}\longmapsto e^{-itx}\in L^{2}\left(\mu\right)
\]
extends to an isometry from $\mathscr{H}\left(K\right)$ into $L^{2}\left(\mu\right)$. 

Let $S\subset\mathbb{R}$ be a countable discrete subset, then for
$f\in\mathscr{H}\left(K\right)$, we have 
\begin{gather*}
\left\langle K_{s},f\right\rangle _{\mathscr{H}\left(K\right)}=0,\;\forall s\in S\\
\Updownarrow\\
\left\langle VK_{s},Vf\right\rangle _{L^{2}\left(\mu\right)}=0,\;\forall s\in S\\
\Updownarrow\\
\int_{\mathbb{R}}e^{isx}\left(Vf\right)\left(x\right)d\mu\left(x\right)=0,\;\forall s\in S.
\end{gather*}
So $S$ has the sampling property if and only if 
\[
\left[\left(\left(Vf\right)d\mu\right)^{\wedge}\left(s\right)=0,\;\forall s\in S\right]\Longrightarrow\begin{bmatrix}Vf=0,\;i.e.,\;f=0,\;\mu-\text{a.e.}\end{bmatrix}
\]
\end{rem}

\section{Discrete RKHSs}

A closely related question from the above discussion is the dichotomy
of \emph{discrete} vs \emph{continuous} RKHSs. Our focus in the present
section is on the discrete case, i.e., RKHSs of functions defined
on a prescribed countable infinite discrete set $V$. 
\begin{defn}[\cite{MR3450534}]
\label{def:dmp}The RKHS $\mathscr{H}=\mathscr{H}\left(K\right)$
is said to have the \emph{discrete mass} property ($\mathscr{H}$
is called a \emph{discrete RKHS}), if $\delta_{x}\in\mathscr{H}$,
for all $x\in V$. Here, $\delta_{x}\left(y\right)$ is the Dirac
mass at $x\in V$. 
\end{defn}

\begin{question}
Let $K:\mathbb{R}^{d}\times\mathbb{R}^{d}\rightarrow\mathbb{R}$ be
positive definite, and let $V\subset\mathbb{R}^{d}$ be a countable
discrete subset. When does $K\big|_{V\times V}$ have the discrete
mass property? 
\end{question}

Of the examples and applications where this question plays an important
role, we emphasize three: (i) discrete Brownian motion-Hilbert spaces,
i.e., discrete versions of the Cameron-Martin Hilbert space; (ii)
energy-Hilbert spaces corresponding to graph-Laplacians; and finally
(iii) RKHSs generated by binomial coefficients. We show that the point-masses
have finite $\mathscr{H}$-norm in cases (i) and (ii), but not in
case (iii).
\begin{defn}
\label{def:d1}Let $V$ be a countably infinite set, and let $\mathscr{F}\left(V\right)$
denote the set of all \emph{finite} subsets of $V$. 
\begin{enumerate}
\item For all $x\in V$, set 
\begin{equation}
K_{x}:=K\left(\cdot,x\right):V\rightarrow\mathbb{C}.\label{eq:pd2}
\end{equation}
\item Let $\mathscr{H}:=\mathscr{H}\left(K\right)$ be the Hilbert-completion
of the $span\left\{ K_{x}:x\in V\right\} $, with respect to the inner
product 
\begin{equation}
\left\langle \sum c_{x}K_{x},\sum d_{y}K_{y}\right\rangle _{\mathscr{H}}:=\sum\sum\overline{c_{x}}d_{y}K\left(x,y\right)\label{eq:pd3}
\end{equation}
$\mathscr{H}$ is then a reproducing kernel Hilbert space (RKHS),
with the reproducing property:
\begin{equation}
\varphi\left(x\right)=\left\langle K_{x},\varphi\right\rangle _{\mathscr{H}},\;\forall x\in V,\:\forall\varphi\in\mathscr{H}.\label{eq:pd31}
\end{equation}
\item If $F\in\mathscr{F}\left(V\right)$, set $\mathscr{H}_{F}=span\{K_{x}\}_{x\in F}\subset\mathscr{H}$,
and let 
\begin{equation}
P_{F}:=\text{the orthogonal projection onto \ensuremath{\mathscr{H}_{F}}}.\label{eq:pd4}
\end{equation}
\item For $F\in\mathscr{F}\left(V\right)$, let $K_{F}$ be the matrix given
by 
\begin{equation}
K_{F}:=\left(K\left(x,y\right)\right)_{\left(x,y\right)\in F\times F}.\label{eq:pd5}
\end{equation}
\end{enumerate}
\end{defn}

\begin{lem}
\label{lem:proj1}Let $F\in\mathscr{F}\left(V\right)=$ all finite
subsets of $V$, $x_{1}\in F$. Assume $\delta_{x_{1}}\in\mathscr{H}$.
Then 
\begin{equation}
P_{F}\left(\delta_{x_{1}}\right)\left(\cdot\right)=\sum_{y\in F}\left(K_{F}^{-1}\delta_{x_{1}}\right)\left(y\right)K_{y}\left(\cdot\right).\label{eq:pd6}
\end{equation}
\end{lem}

\begin{proof}
Show that
\begin{equation}
\delta_{x_{1}}-\sum_{y\in F}\left(K_{F}^{-1}\delta_{x_{1}}\right)\left(y\right)K_{y}\left(\cdot\right)\in\mathscr{H}_{F}^{\perp}.\label{eq:pd7}
\end{equation}
The remaining part follows easily from this. 
\end{proof}
\begin{thm}
\label{thm:del}Let $V$ and $K$ as above, i.e., we assume that $V$
is countably infinite, and $K$ is a p.d. kernel on $V\times V$.
Let $\mathscr{H}=\mathscr{H}\left(K\right)$ be the corresponding
RKHS. Fix $x_{1}\in V$. Then the following three conditions are equivalent:

\begin{enumerate}
\item \label{enu:d1}$\delta_{x_{1}}\in\mathscr{H}$; 
\item \label{enu:d2}$\exists C_{x_{1}}<\infty$ such that for all $F\in\mathscr{F}\left(V\right)$,
we have
\begin{equation}
\left|\xi\left(x_{1}\right)\right|^{2}\leq C_{x_{1}}\underset{F\times F}{\sum\sum}\overline{\xi\left(x\right)}\xi\left(y\right)K\left(x,y\right).\label{eq:d1}
\end{equation}
\item \label{enu:d3}For $F\in\mathscr{F}\left(V\right)$, set 
\begin{equation}
K_{F}=\left(K\left(x,y\right)\right)_{\left(x,y\right)\in F\times F}\label{eq:d2}
\end{equation}
as a $\#F\times\#F$ matrix. Then
\begin{equation}
\sup_{F\in\mathscr{F}\left(V\right)}\left(K_{F}^{-1}\delta_{x_{1}}\right)\left(x_{1}\right)<\infty.\label{eq:d3}
\end{equation}
\end{enumerate}
\end{thm}

\begin{proof}
This is an application of \remref{rk}. Also see \cite{MR3450534}
for details.
\end{proof}
Let $D$ be an open domain in $\mathbb{R}^{d}$, and assume $V\subset D$
is countable and discrete subset of $D$. In this case, we shall consider
two positive definite kernels: the original kernel $K$ on $D\times D$,
and $K_{V}:=K\big|_{V\times V}$ on $V\times V$ by restriction. Thus
if $x\in V$, then 
\[
K_{x}^{\left(V\right)}\left(\cdot\right)=K\left(\cdot,x\right):V\longrightarrow\mathbb{R}
\]
is a function on $V$, while 
\[
K_{x}\left(\cdot\right):=K\left(\cdot,x\right):D\longrightarrow\mathbb{R}
\]
is a function on $D$. Further, let $\mathscr{H}$ and $\mathscr{H}_{V}$
be the associated RKHSs respectively. 
\begin{lem}
\label{lem:mc1}$\mathscr{H}_{V}$ is isometrically embedded into
$\mathscr{H}$ via the mapping
\[
J^{\left(V\right)}:K_{x}^{\left(V\right)}\longmapsto K_{x},\;x\in V.
\]
\end{lem}

\begin{proof}
Assume $F\in\mathscr{F}\left(V\right)$, i.e., $F$ is a finite subset
of $V$. Let $\xi=\xi_{F}$ is a function on $F$, then 
\[
\left\Vert \sum\nolimits _{x\in F}\xi\left(x\right)K_{x}^{\left(V\right)}\right\Vert _{\mathscr{H}_{V}}=\left\Vert \sum\nolimits _{x\in F}\xi\left(x\right)K_{x}\right\Vert _{\mathscr{H}}.
\]

Note that, by definition, the linear span of $\{K_{x}^{\left(V\right)}\mathrel{;}x\in V\}$
is dense in $\mathscr{H}_{V}$, and the span of $\{K_{x}\mathrel{;}x\in D\}$
is dense in $\mathscr{H}$. We conclude that $J^{\left(V\right)}$
extends uniquely to an isometry from $\mathscr{H}_{V}$ into $\mathscr{H}$.
 The desired result follows from this. 
\end{proof}
In the examples below, we are concerned with cases of kernels $K:D\times D\rightarrow\mathbb{R}$
with restriction $K_{V}:V\times V\rightarrow\mathbb{R}$, where $V$
is a countable discrete subset of $D$. Typically, for $x\in V$,
we may have the restriction $\delta_{x}\big|_{V}$ contained in $\mathscr{H}_{V}$,
but $\delta_{x}$ in not in $\mathscr{H}$. 

\subsection{\label{subsec:bm}Brownian Motion}

Consider the covariance function of standard Brownian motion $B_{t}$,
$t\in[0,\infty)$, i.e., a Gaussian process $\left\{ B_{t}\right\} $
with mean zero and covariance function 
\begin{equation}
K\left(s,t\right):=\mathbb{E}\left(B_{s}B_{t}\right)=s\wedge t.\label{eq:bm1}
\end{equation}

Restrict to $V:=\left\{ 0\right\} \cup\mathbb{Z}_{+}\subset D$, i.e.,
consider 
\[
K^{\left(V\right)}=K\big|_{V\times V}.
\]
$\mathscr{H}\left(K\right)$: Cameron-Martin Hilbert space, consisting
of functions $f\in L^{2}\left(\mathbb{R}\right)$ s.t. 
\[
\int_{0}^{\infty}\left|f'\left(x\right)\right|^{2}dx<\infty,\quad f\left(0\right)=0.
\]
$\mathscr{H}_{V}:=\mathscr{H}\left(K_{V}\right)$. Note that 
\[
f\in\mathscr{H}\left(K_{V}\right)\Longleftrightarrow\sum_{n}\left|f\left(n\right)-f\left(n+1\right)\right|^{2}<\infty.
\]

We now show that the restriction of (\ref{eq:bm1}) to $V\times V$
for an ordered subset (we fix such a set $V$):
\begin{equation}
V:\;0<x_{1}<x_{2}<\cdots<x_{i}<x_{i+1}<\cdots\label{eq:bm2}
\end{equation}
has the discrete mass property (\defref{dmp}). 

Set $\mathscr{H}_{V}=RKHS(K\big|_{V\times V})$, 
\begin{equation}
K_{V}\left(x_{i},x_{j}\right)=x_{i}\wedge x_{j}.\label{eq:bm3}
\end{equation}
We consider the set $F_{n}=\left\{ x_{1},x_{2},\ldots,x_{n}\right\} $
of finite subsets of $V$, and 
\begin{equation}
K_{n}=K^{\left(F_{n}\right)}=\begin{bmatrix}x_{1} & x_{1} & x_{1} & \cdots & x_{1}\\
x_{1} & x_{2} & x_{2} & \cdots & x_{2}\\
x_{1} & x_{2} & x_{3} & \cdots & x_{3}\\
\vdots & \vdots & \vdots & \vdots & \vdots\\
x_{1} & x_{2} & x_{3} & \cdots & x_{n}
\end{bmatrix}=\left(x_{i}\wedge x_{j}\right)_{i,j=1}^{n}.\label{eq:bm4}
\end{equation}
We will show that condition (\ref{enu:d3}) in \thmref{del} holds
for $K_{V}$. 
\begin{lem}
~ 
\begin{equation}
D_{n}=\det\left(\left(x_{i}\wedge x_{j}\right)_{i,j=1}^{n}\right)=x_{1}\left(x_{2}-x_{1}\right)\left(x_{3}-x_{2}\right)\cdots\left(x_{n}-x_{n-1}\right).\label{eq:bm5}
\end{equation}
\end{lem}

\begin{proof}
Induction. In fact, 
\[
\begin{bmatrix}x_{1} & x_{1} & x_{1} & \cdots & x_{1}\\
x_{1} & x_{2} & x_{2} & \cdots & x_{2}\\
x_{1} & x_{2} & x_{3} & \cdots & x_{3}\\
\vdots & \vdots & \vdots & \vdots & \vdots\\
x_{1} & x_{2} & x_{3} & \cdots & x_{n}
\end{bmatrix}\sim\begin{bmatrix}x_{1} & 0 & 0 & \cdots & 0\\
0 & x_{2}-x_{1} & 0 & \cdots & 0\\
0 & 0 & x_{3}-x_{2} & \cdots & 0\\
\vdots & \vdots & \vdots & \ddots & \vdots\\
0 & \cdots & 0 & \cdots & x_{n}-x_{n-1}
\end{bmatrix},
\]
unitary equivalence in finite dimensions.
\end{proof}

\begin{lem}
Let 
\begin{equation}
\zeta_{\left(n\right)}:=K_{n}^{-1}\left(\delta_{x_{1}}\right)\left(\cdot\right)\label{eq:bm7}
\end{equation}
so that 
\begin{equation}
\left\Vert P_{F_{n}}\left(\delta_{x_{1}}\right)\right\Vert _{\mathscr{H}_{V}}^{2}=\zeta_{\left(n\right)}\left(x_{1}\right).\label{eq:bm8}
\end{equation}
Then, 
\begin{align*}
\zeta_{\left(1\right)}\left(x_{1}\right) & =\frac{1}{x_{1}}\\
\zeta_{\left(n\right)}\left(x_{1}\right) & =\frac{x_{2}}{x_{1}\left(x_{2}-x_{1}\right)},\quad\text{for}\;n=2,3,\ldots,
\end{align*}
and 
\[
\left\Vert \delta_{x_{1}}\right\Vert _{\mathscr{H}_{V}}^{2}=\frac{x_{2}}{x_{1}\left(x_{2}-x_{1}\right)}.
\]
\end{lem}

\begin{proof}
A direct computation shows the $\left(1,1\right)$ minor of the matrix
$K_{n}^{-1}$ is
\begin{equation}
D'_{n-1}=\det\left(\left(x_{i}\wedge x_{j}\right)_{i,j=2}^{n}\right)=x_{2}\left(x_{3}-x_{2}\right)\left(x_{4}-x_{3}\right)\cdots\left(x_{n}-x_{n-1}\right)\label{eq:bm6}
\end{equation}
and so 
\begin{align*}
\zeta_{\left(1\right)}\left(x_{1}\right) & =\frac{1}{x_{1}},\quad\mbox{and}\\
\zeta_{\left(2\right)}\left(x_{1}\right) & =\frac{x_{2}}{x_{1}\left(x_{2}-x_{1}\right)}\\
\zeta_{\left(3\right)}\left(x_{1}\right) & =\frac{x_{2}\left(x_{3}-x_{2}\right)}{x_{1}\left(x_{2}-x_{1}\right)\left(x_{3}-x_{2}\right)}=\frac{x_{2}}{x_{1}\left(x_{2}-x_{1}\right)}\\
\zeta_{\left(4\right)}\left(x_{1}\right) & =\frac{x_{2}\left(x_{3}-x_{2}\right)\left(x_{4}-x_{3}\right)}{x_{1}\left(x_{2}-x_{1}\right)\left(x_{3}-x_{2}\right)\left(x_{4}-x_{3}\right)}=\frac{x_{2}}{x_{1}\left(x_{2}-x_{1}\right)}\\
 & \vdots
\end{align*}
The result follows from this.
\end{proof}
\begin{cor}
\label{cor:proj}$P_{F_{n}}\left(\delta_{x_{1}}\right)=P_{F_{2}}\left(\delta_{x_{1}}\right)$,
$\forall n\geq2$. Therefore, 
\begin{equation}
\delta_{x_{1}}\in\mathscr{H}_{V}^{\left(F_{2}\right)}:=span\{K_{x_{1}}^{\left(V\right)},K_{x_{2}}^{\left(V\right)}\}
\end{equation}
and
\begin{equation}
\delta_{x_{1}}=\zeta_{\left(2\right)}\left(x_{1}\right)K_{x_{1}}^{\left(V\right)}+\zeta_{\left(2\right)}\left(x_{2}\right)K_{x_{2}}^{\left(V\right)}
\end{equation}
where 
\[
\zeta_{\left(2\right)}\left(x_{i}\right)=K_{2}^{-1}\left(\delta_{x_{1}}\right)\left(x_{i}\right),\;i=1,2.
\]
Specifically, 
\begin{align}
\zeta_{\left(2\right)}\left(x_{1}\right) & =\frac{x_{2}}{x_{1}\left(x_{2}-x_{1}\right)}\\
\zeta_{\left(2\right)}\left(x_{2}\right) & =\frac{-1}{x_{2}-x_{1}};
\end{align}
and 
\begin{equation}
\left\Vert \delta_{x_{1}}\right\Vert _{\mathscr{H}_{V}}^{2}=\frac{x_{2}}{x_{1}\left(x_{2}-x_{1}\right)}.\label{eq:dn}
\end{equation}
\end{cor}

\begin{proof}
Note that 
\[
\zeta_{n}\left(x_{1}\right)=\left\Vert P_{F_{n}}\left(\delta_{x_{1}}\right)\right\Vert _{\mathscr{H}}^{2}
\]
and $\zeta_{\left(1\right)}\left(x_{1}\right)\leq\zeta_{\left(2\right)}\left(x_{1}\right)\leq\cdots$,
since $F_{n}=\left\{ x_{1},x_{2},\ldots,x_{n}\right\} $. In particular,
$\frac{1}{x_{1}}\leq\frac{x_{2}}{x_{1}\left(x_{2}-x_{1}\right)}$,
which yields (\ref{eq:dn}). 
\end{proof}
\begin{rem}
We showed that $\delta_{x_{1}}\in\mathscr{H}_{V}$, $V=\left\{ x_{1}<x_{2}<\cdots\right\} \subset\mathbb{R}_{+}$,
with the restriction of $s\wedge t$ = the covariance kernel of Brownian
motion. The same argument also shows that $\delta_{x_{i}}\in\mathscr{H}_{V}$
when $i>1$.

Conclusions: 
\begin{align}
\delta_{x_{i}} & \in span\left\{ K_{x_{i-1}}^{\left(V\right)},K_{x_{i}}^{\left(V\right)},K_{x_{i+1}}^{\left(V\right)}\right\} ,\quad\mbox{and}\\
\left\Vert \delta_{x_{i}}\right\Vert _{\mathscr{H}}^{2} & =\frac{x_{i+1}-x_{i-1}}{\left(x_{i}-x_{i-1}\right)\left(x_{i+1}-x_{i}\right)}.
\end{align}
Details are left for the interested readers.
\end{rem}

\begin{cor}
Let $V\subset\mathbb{R}_{+}$ be countable. If $x_{a}\in V$ is an
accumulation point (from $V$), then $\left\Vert \delta_{a}\right\Vert _{\mathscr{H}_{V}}=\infty$. 
\end{cor}

\begin{example}[Sparse sample-points]
Let $V=\left\{ x_{i}\right\} _{i=1}^{\infty}$, where 
\[
x_{i}=\frac{i\left(i-1\right)}{2},\quad i\in\mathbb{N}.
\]
It follows that $x_{i+1}-x_{i}=i$, and so 
\[
\left\Vert \delta_{x_{i}}\right\Vert _{\mathscr{H}}^{2}=\frac{x_{i+1}-x_{i}}{\left(x_{i}-x_{i-1}\right)\left(x_{i+1}-x_{i}\right)}=\frac{2i-1}{\left(i-1\right)i}\xrightarrow[i\rightarrow\infty]{}0.
\]
We conclude that $\left\Vert \delta_{x_{i}}\right\Vert _{\mathscr{H}}\xrightarrow[i\rightarrow\infty]{}0$
if the set $V=\left\{ x_{i}\right\} _{i=1}^{\infty}\subset\mathbb{R}_{+}$
is sparse. 
\end{example}

Now, some general facts:
\begin{lem}
Let $K:V\times V\rightarrow\mathbb{C}$ be p.d., and let $\mathscr{H}$
be the corresponding RKHS. If $x_{1}\in V$, and if $\delta_{x_{1}}$
has a representation as follows:
\begin{equation}
\delta_{x_{1}}=\sum_{y\in V}\zeta^{\left(x_{1}\right)}\left(y\right)K_{y}\;,\label{eq:pr1}
\end{equation}
then
\begin{equation}
\left\Vert \delta_{x_{1}}\right\Vert _{\mathscr{H}}^{2}=\zeta^{\left(x_{1}\right)}\left(x_{1}\right).\label{eq:pr2}
\end{equation}
\end{lem}

\begin{proof}
Substitute both sides of (\ref{eq:pr1}) into $\left\langle \delta_{x_{1}},\cdot\right\rangle _{\mathscr{H}}$
where $\left\langle \cdot,\cdot\right\rangle _{\mathscr{H}}$ denotes
the inner product in $\mathscr{H}$. 
\end{proof}

\subsection{Brownian Bridge}

Let $D:=\left(0,1\right)=$ the open interval $0<t<1$, and set 
\begin{equation}
K_{bridge}\left(s,t\right):=s\wedge t-st;\label{eq:bb1}
\end{equation}
then (\ref{eq:bb1}) is the covariance function for the Brownian bridge
$B_{bri}\left(t\right)$, i.e., 
\begin{equation}
B_{bri}\left(0\right)=B_{bri}\left(1\right)=0\label{eq:bb2}
\end{equation}

\begin{figure}[H]
\includegraphics[width=0.35\columnwidth]{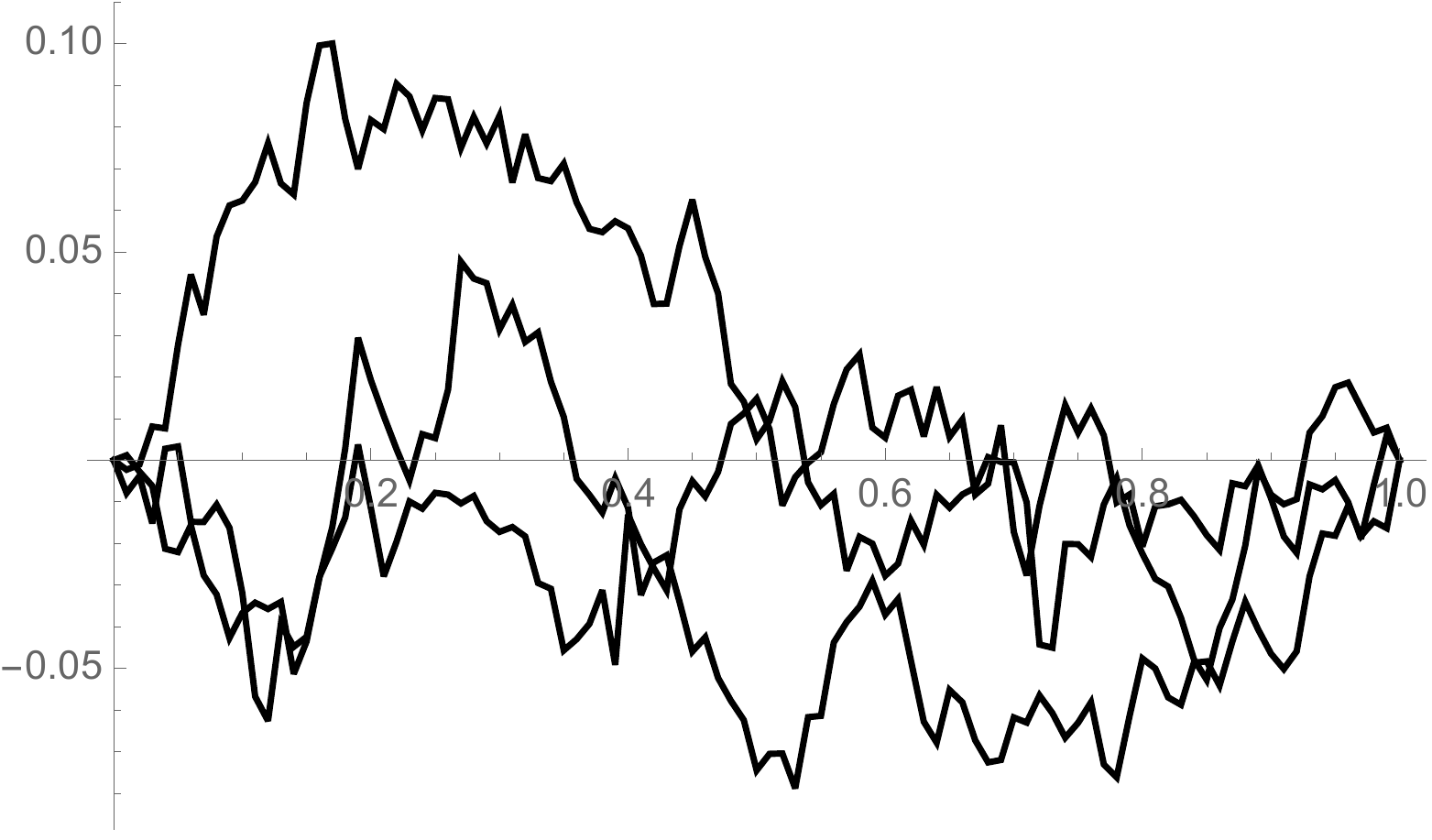}

\caption{\label{fig:bb}Brownian bridge $B_{bri}\left(t\right)$, a simulation
of three sample paths of the Brownian bridge.}
\end{figure}

\begin{equation}
B_{bri}\left(t\right)=\left(1-t\right)B\left(\frac{t}{1-t}\right),\quad0<t<1;\label{eq:bb3}
\end{equation}
where $B\left(t\right)$ is Brownian motion; see \lemref{mc1}.

The corresponding Cameron-Martin space is now 
\begin{equation}
\mathscr{H}_{bri}=\left\{ f\;\mbox{on}\:\left[0,1\right];f'\in L^{2}\left(0,1\right),f\left(0\right)=f\left(1\right)=0\right\} \label{eq:bb4}
\end{equation}
with 
\begin{equation}
\left\Vert f\right\Vert _{\mathscr{H}_{bri}}^{2}:=\int_{0}^{1}\left|f'\left(s\right)\right|^{2}ds<\infty.\label{eq:bb5}
\end{equation}

If $V=\left\{ x_{i}\right\} _{i=1}^{\infty}$, $x_{1}<x_{2}<\cdots<1$,
is the discrete subset of $D$, then we have for $F_{n}\in\mathscr{F}\left(V\right)$,
$F_{n}=\left\{ x_{1},x_{2},\cdots,x_{n}\right\} $, 
\begin{equation}
K_{F_{n}}=\left(K_{bridge}\left(x_{i},x_{j}\right)\right)_{i,j=1}^{n},\label{eq:bb6}
\end{equation}
see (\ref{eq:bb1}), and 
\begin{equation}
\det K_{F_{n}}=x_{1}\left(x_{2}-x_{1}\right)\cdots\left(x_{n}-x_{n-1}\right)\left(1-x_{n}\right).\label{eq:bb7}
\end{equation}

As a result, we get $\delta_{x_{i}}\in\mathscr{H}_{V}^{\left(bri\right)}$
for all $i$, and 
\[
\left\Vert \delta_{x_{i}}\right\Vert _{\mathscr{H}_{V}^{\left(bri\right)}}^{2}=\frac{x_{i+1}-x_{i-1}}{\left(x_{i+1}-x_{i}\right)\left(x_{i}-x_{i-1}\right)}.
\]
Note $\lim_{x_{i}\rightarrow1}\left\Vert \delta_{x_{i}}\right\Vert _{\mathscr{H}_{V}^{\left(bri\right)}}^{2}=\infty$. 

\subsection{Binomial RKHS}

The purpose of the present subsection if to display a concrete RKHS
$\mathscr{H}\left(K\right)$ in the discrete framework with the property
that $\mathscr{H}\left(K\right)$ does not contain the Dirac masses
$\delta_{x}$. The RKHS in question is generated by the binomial coefficients,
and it is relevant for a host of applications; see e.g., \cite{MR3484546,MR3228856,MR1822449}.
\begin{defn}
Let $V=\mathbb{Z}_{+}\cup\left\{ 0\right\} $; and 
\[
K_{b}\left(x,y\right):=\sum_{n=0}^{x\wedge y}\binom{x}{n}\binom{y}{n},\quad\left(x,y\right)\in V\times V.
\]
where $\binom{x}{n}=\frac{x\left(x-1\right)\cdots\left(x-n+1\right)}{n!}$
denotes the standard binomial coefficient from the binomial expansion.

Let $\mathscr{H}=\mathscr{H}\left(K_{b}\right)$ be the corresponding
RKHS. Set 
\begin{equation}
e_{n}\left(x\right)=\begin{cases}
\binom{x}{n} & \text{if \ensuremath{n\leq x}}\\
0 & \text{if \ensuremath{n>x}}.
\end{cases}\label{eq:b1}
\end{equation}
\end{defn}

\begin{lem}[\cite{MR3367659}]
\label{lem:b1}~

\begin{enumerate}
\item $e_{n}\left(\cdot\right)\in\mathscr{H}$, $n\in V$; 
\item $\left\{ e_{n}\right\} _{n\in V}$ is an orthonormal basis (ONB) in
the Hilbert space $\mathscr{H}$. 
\item Given $f\in\mathscr{F}unc\left(V\right)$; then 
\begin{equation}
f\in\mathscr{H}\Longleftrightarrow\sum_{k=0}^{\infty}\left|\left\langle e_{k},f\right\rangle _{\mathscr{H}}\right|^{2}<\infty;\label{eq:b4}
\end{equation}
and, in this case, 
\[
\left\Vert f\right\Vert _{\mathscr{H}}^{2}=\sum_{k=0}^{\infty}\left|\left\langle e_{k},f\right\rangle _{\mathscr{H}}\right|^{2}.
\]
\item Set $F_{n}=\left\{ 0,1,2,\ldots,n\right\} $, and 
\begin{equation}
P_{F_{n}}=\sum_{k=0}^{n}\left|e_{k}\left\rangle \right\langle e_{k}\right|\label{eq:b2}
\end{equation}
or equivalently 
\begin{equation}
P_{F_{n}}f=\sum_{k=0}^{n}\left\langle e_{k},f\right\rangle _{\mathscr{H}}e_{k}\,.\label{eq:b3}
\end{equation}
Then formula (\ref{eq:b3}) is well defined for all functions $f\in\mathscr{F}unc\left(V\right)$. 
\end{enumerate}
\end{lem}

Fix $x_{1}\in V$, then we shall apply \lemref{b1} to the function
$f_{1}=\delta_{x_{1}}$ (in $\mathscr{F}unc\left(V\right)$). 
\begin{thm}
\label{thm:bino}We have 
\[
\left\Vert P_{F_{n}}\left(\delta_{x_{1}}\right)\right\Vert _{\mathscr{H}}^{2}=\sum_{k=x_{1}}^{n}\binom{k}{x_{1}}^{2}.
\]
\end{thm}

The proof of the theorem will be subdivided in steps; see below. 
\begin{lem}[\cite{MR3367659}]
~

\begin{enumerate}
\item \label{enu:b1}For $\forall m,n\in V$, such that $m\leq n$, we have
\begin{equation}
\delta_{m,n}=\sum_{j=m}^{n}\left(-1\right)^{m+j}\binom{n}{j}\binom{j}{m}.\label{eq:b5}
\end{equation}
\item \label{enu:b2}For all $n\in\mathbb{Z}_{+}$, the inverse of the following
lower triangle matrix is this: With (see Figure \ref{fig:L}) 
\begin{equation}
L_{xy}^{\left(n\right)}=\begin{cases}
\binom{x}{y} & \text{if \ensuremath{y\leq x\leq n}}\\
0 & \text{if \ensuremath{x<y}}
\end{cases}\label{eq:b6}
\end{equation}
 we have:
\begin{equation}
\left(L^{\left(n\right)}\right)_{xy}^{-1}=\begin{cases}
\left(-1\right)^{x-y}\binom{x}{y} & \text{if \ensuremath{y\leq x\leq n}}\\
0 & \text{if \ensuremath{x<y}}.
\end{cases}\label{eq:b7}
\end{equation}
 
\end{enumerate}
Notation: The numbers in (\ref{eq:b7}) are the entries of the matrix
$\left(L^{\left(n\right)}\right)^{-1}$. 

\end{lem}

\begin{proof}
In rough outline, (\ref{enu:b2}) follows from (\ref{enu:b1}).
\end{proof}
\begin{figure}[H]
\[
L^{\left(n\right)}=\begin{bmatrix}1 & 0 & 0 & 0 & \cdots & \cdots & 0 & \cdots & 0 & 0\\
1 & 1 & 0 & 0 & \cdots & \cdots & 0 & \cdots & 0 & 0\\
1 & 2 & 1 & 0 &  &  & \vdots &  & \vdots & \vdots\\
1 & 3 & 3 & 1 & \ddots &  & \vdots &  & \vdots & \vdots\\
\vdots & \vdots & \vdots & \vdots & \ddots &  & \vdots &  & \vdots & \vdots\\
\vdots & \vdots & \vdots & \vdots &  & 1 & 0 &  & \vdots & \vdots\\
1 & \cdots & \binom{x}{y} & \binom{x}{y+1} & \cdots & * & 1 & \ddots & \vdots & \vdots\\
\vdots & \vdots & \vdots & \vdots &  &  &  & \ddots & 0 & \vdots\\
\vdots & \vdots & \vdots & \vdots &  &  &  &  & 1 & 0\\
1 & \cdots & \binom{n}{y} & \binom{n}{y+1} & \cdots & \cdots & \cdots & \cdots & n & 1
\end{bmatrix}
\]

\caption{\label{fig:L}The matrix $L_{n}$ is simply a truncated Pascal triangle,
arranged to fit into a lower triangular matrix.}
\end{figure}
\begin{cor}
\label{cor:bino}Let $K_{b}$, $\mathscr{H}$, and $n\in\mathbb{Z}_{+}$
be as above with the lower triangle matrix $L_{n}$. Set 
\begin{equation}
K_{n}\left(x,y\right)=K_{b}\left(x,y\right),\quad\left(x,y\right)\in F_{n}\times F_{n},\label{eq:b8}
\end{equation}
i.e., an $\left(n+1\right)\times\left(n+1\right)$ matrix. 

\begin{enumerate}
\item Then $K_{n}$ is invertible with 
\begin{equation}
K_{n}^{-1}=\left(L_{n}^{tr}\right)^{-1}\left(L_{n}\right)^{-1};\label{eq:b9}
\end{equation}
an $(\text{upper triangle})\times(\text{lower triangle})$ factorization. 
\item For the diagonal entries in the $\left(n+1\right)\times\left(n+1\right)$
matrix $K_{n}^{-1}$, we have:
\[
\left\langle x,K_{n}^{-1}x\right\rangle _{l^{2}}=\sum_{k=x}^{n}\binom{k}{x}^{2}
\]
\end{enumerate}
Conclusion\textbf{:} Since 
\begin{equation}
\left\Vert P_{F_{n}}\left(\delta_{x_{1}}\right)\right\Vert _{\mathscr{H}}^{2}=\left\langle x_{1},K_{n}^{-1}x_{1}\right\rangle _{\mathscr{H}}\label{eq:b11}
\end{equation}
for all $x_{1}\in F_{n}$, we get 
\begin{align}
\left\Vert P_{F_{n}}\left(\delta_{x_{1}}\right)\right\Vert _{\mathscr{H}}^{2} & =\sum_{k=x_{1}}^{n}\binom{k}{x_{1}}^{2}\nonumber \\
 & =1+\binom{x_{1}+1}{x_{1}}^{2}+\binom{x_{1}+2}{x_{1}}^{2}+\cdots+\binom{n}{x_{1}}^{2};\label{eq:b12}
\end{align}
and therefore, 
\[
\left\Vert \delta_{x_{1}}\right\Vert _{\mathscr{H}}^{2}=\sum_{k=x_{1}}^{\infty}\binom{k}{x_{1}}^{2}=\infty.
\]
In other words, no $\delta_{x}$ is in $\mathscr{H}$.
\end{cor}

\begin{acknowledgement*}
The co-authors thank the following colleagues for helpful and enlightening
discussions: Professors Daniel Alpay, Sergii Bezuglyi, Ilwoo Cho,
Myung-Sin Song, Wayne Polyzou, and members in the Math Physics seminar
at The University of Iowa.
\end{acknowledgement*}
\bibliographystyle{amsalpha}
\bibliography{ref}

\newcommand{\etalchar}[1]{$^{#1}$}
\providecommand{\bysame}{\leavevmode\hbox to3em{\hrulefill}\thinspace}
\providecommand{\MR}{\relax\ifhmode\unskip\space\fi MR }
\providecommand{\MRhref}[2]{%
  \href{http://www.ams.org/mathscinet-getitem?mr=#1}{#2}
}
\providecommand{\href}[2]{#2}
\begin{thebibliography}{ABDdS93}

\bibitem[AACM11]{MR2868037}
Akram Aldroubi, Magal{\'{\i}} Anastasio, Carlos Cabrelli, and Ursula Molter,
  \emph{A dimension reduction scheme for the computation of optimal unions of
  subspaces}, Sampl. Theory Signal Image Process. \textbf{10} (2011), no.~1-2,
  135--150. \MR{2868037}

\bibitem[ABDdS93]{ABDdS93}
Daniel Alpay, Vladimir Bolotnikov, Aad Dijksma, and Henk de~Snoo, \emph{On some
  operator colligations and associated reproducing kernel {H}ilbert spaces},
  Operator extensions, interpolation of functions and related topics, Oper.
  Theory Adv. Appl., vol.~61, Birkh\"auser, Basel, 1993, pp.~1--27. \MR{1246577
  (94i:47018)}

\bibitem[ACH{\etalchar{+}}10]{MR2587581}
Akram Aldroubi, Carlos Cabrelli, Christopher Heil, Keri Kornelson, and Ursula
  Molter, \emph{Invariance of a shift-invariant space}, J. Fourier Anal. Appl.
  \textbf{16} (2010), no.~1, 60--75. \MR{2587581}

\bibitem[AD92]{AD92}
Daniel Alpay and Harry Dym, \emph{On reproducing kernel spaces, the {S}chur
  algorithm, and interpolation in a general class of domains}, Operator theory
  and complex analysis ({S}apporo, 1991), Oper. Theory Adv. Appl., vol.~59,
  Birkh\"auser, Basel, 1992, pp.~30--77. \MR{1246809 (94j:46034)}

\bibitem[AD93]{AD93}
\bysame, \emph{On a new class of structured reproducing kernel spaces}, J.
  Funct. Anal. \textbf{111} (1993), no.~1, 1--28. \MR{1200633 (94g:46035)}

\bibitem[AJ14]{MR3367659}
Daniel Alpay and Palle Jorgensen, \emph{Reproducing kernel {H}ilbert spaces
  generated by the binomial coefficients}, Illinois J. Math. \textbf{58}
  (2014), no.~2, 471--495. \MR{3367659}

\bibitem[AJ15]{MR3402823}
\bysame, \emph{Spectral theory for {G}aussian processes: reproducing kernels,
  boundaries, and {$L^2$}-wavelet generators with fractional scales}, Numer.
  Funct. Anal. Optim. \textbf{36} (2015), no.~10, 1239--1285. \MR{3402823}

\bibitem[AJSV13]{AJSV13}
Daniel Alpay, Palle Jorgensen, Ron Seager, and Dan Volok, \emph{On discrete
  analytic functions: products, rational functions and reproducing kernels}, J.
  Appl. Math. Comput. \textbf{41} (2013), no.~1-2, 393--426. \MR{3017129}

\bibitem[AJV14]{MR3251728}
Daniel Alpay, Palle Jorgensen, and Dan Volok, \emph{Relative reproducing kernel
  {H}ilbert spaces}, Proc. Amer. Math. Soc. \textbf{142} (2014), no.~11,
  3889--3895. \MR{3251728}

\bibitem[Aro43]{Aro43}
Nachman Aronszajn, \emph{La th\'eorie des noyaux reproduisants et ses
  applications. {I}}, Proc. Cambridge Philos. Soc. \textbf{39} (1943),
  133--153. \MR{0008639 (5,38e)}

\bibitem[Aro48]{Aro48}
\bysame, \emph{Reproducing and pseudo-reproducing kernels and their application
  to the partial differential equations of physics}, Studies in partial
  differential equations. Technical report 5, preliminary note, Harvard
  University, Graduate School of Engineering., 1948. \MR{0031663 (11,187b)}

\bibitem[BCL11]{MR2837145}
Radu Balan, Pete Casazza, and Zeph Landau, \emph{Redundancy for localized
  frames}, Israel J. Math. \textbf{185} (2011), 445--476. \MR{2837145
  (2012i:42036)}

\bibitem[BDP05]{MR2154344}
Stefan Bildea, Dorin~Ervin Dutkay, and Gabriel Picioroaga, \emph{M{RA}
  super-wavelets}, New York J. Math. \textbf{11} (2005), 1--19. \MR{2154344}

\bibitem[CD93]{CoDa93}
Albert Cohen and Ingrid Daubechies, \emph{Non-separable bidimensional wavelet
  bases}, Revista Matem{\'{a}}tica Iberoamericana (1993), 51--137.

\bibitem[Chr14]{MR3167899}
Ole Christensen, \emph{A short introduction to frames, {G}abor systems, and
  wavelet systems}, Azerb. J. Math. \textbf{4} (2014), no.~1, 25--39.
  \MR{3167899}

\bibitem[CS02]{CS02}
Felipe Cucker and Steve Smale, \emph{On the mathematical foundations of
  learning}, Bull. Amer. Math. Soc. (N.S.) \textbf{39} (2002), no.~1, 1--49.
  \MR{1864085 (2003a:68118)}

\bibitem[Dok14]{MR3228856}
Nikolai Dokuchaev, \emph{On strong causal binomial approximation for stochastic
  processes}, Discrete Contin. Dyn. Syst. Ser. B \textbf{19} (2014), no.~6,
  1549--1562. \MR{3228856}

\bibitem[Dut06]{Dutkay_2006}
Dorin~Ervin Dutkay, \emph{Low-pass filters and representations of the baumslag
  solitar group}, Transactions of the American Mathematical Society
  \textbf{358} (2006), no.~12, 5271--5292.

\bibitem[FJKO05]{MR2147063}
Matthew Fickus, Brody~D. Johnson, Keri Kornelson, and Kasso~A. Okoudjou,
  \emph{Convolutional frames and the frame potential}, Appl. Comput. Harmon.
  Anal. \textbf{19} (2005), no.~1, 77--91. \MR{2147063 (2006d:42050)}

\bibitem[Gal01]{MR1822449}
Roza Galeeva, \emph{Binomial trees as dynamical systems}, Phys. A \textbf{292}
  (2001), no.~1-4, 519--535. \MR{1822449}

\bibitem[Hid80]{Hi80}
Takeyuki Hida, \emph{Brownian motion}, Applications of Mathematics, vol.~11,
  Springer-Verlag, New York, 1980, Translated from the Japanese by the author
  and T. P. Speed. \MR{562914 (81a:60089)}

\bibitem[HKL{\etalchar{+}}14]{HKL14}
S.~{Haeseler}, M.~{Keller}, D.~{Lenz}, J.~{Masamune}, and M.~{Schmidt},
  \emph{{Global properties of Dirichlet forms in terms of Green's formula}},
  ArXiv e-prints (2014).

\bibitem[HKLW07]{MR2367342}
Deguang Han, Keri Kornelson, David Larson, and Eric Weber, \emph{Frames for
  undergraduates}, Student Mathematical Library, vol.~40, American Mathematical
  Society, Providence, RI, 2007. \MR{2367342 (2010e:42044)}

\bibitem[HLS85]{MR799420}
J.~M. Harrison, H.~J. Landau, and L.~A. Shepp, \emph{The stationary
  distribution of reflected {B}rownian motion in a planar region}, Ann. Probab.
  \textbf{13} (1985), no.~3, 744--757. \MR{799420}

\bibitem[HN14]{MR3201917}
Haakan Hedenmalm and Pekka~J. Nieminen, \emph{The {G}aussian free field and
  {H}adamard's variational formula}, Probab. Theory Related Fields \textbf{159}
  (2014), no.~1-2, 61--73. \MR{3201917}

\bibitem[HQKL10]{MR2607639}
Minh Ha~Quang, Sung~Ha Kang, and Triet~M. Le, \emph{Image and video
  colorization using vector-valued reproducing kernel {H}ilbert spaces}, J.
  Math. Imaging Vision \textbf{37} (2010), no.~1, 49--65. \MR{2607639
  (2011k:94032)}

\bibitem[JKS16]{MR3484546}
Norman Josephy, Lucia Kimball, and Victoria Steblovskaya, \emph{Optimal hedging
  in an extended binomial market under transaction costs}, Quant. Finance
  \textbf{16} (2016), no.~5, 763--776. \MR{3484546}

\bibitem[JT15a]{2015arXiv150202549J}
P.~{Jorgensen} and F.~{Tian}, \emph{{Infinite weighted graphs with bounded
  resistance metric}}, ArXiv e-prints (2015).

\bibitem[JT15b]{MR3450534}
Palle Jorgensen and Feng Tian, \emph{Discrete reproducing kernel {H}ilbert
  spaces: sampling and distribution of {D}irac-masses}, J. Mach. Learn. Res.
  \textbf{16} (2015), 3079--3114. \MR{3450534}

\bibitem[JT16]{zbMATH06664785}
Palle {Jorgensen} and Feng {Tian}, \emph{{Positive definite kernels and
  boundary spaces.}}, {Adv. Oper. Theory} \textbf{1} (2016), no.~1, 123--133
  (English).

\bibitem[Kat04]{MR2039503}
Yitzhak Katznelson, \emph{An introduction to harmonic analysis}, third ed.,
  Cambridge Mathematical Library, Cambridge University Press, Cambridge, 2004.
  \MR{2039503}

\bibitem[KH11]{KH11}
Sanjeev Kulkarni and Gilbert Harman, \emph{An elementary introduction to
  statistical learning theory}, Wiley Series in Probability and Statistics,
  John Wiley \& Sons, Inc., Hoboken, NJ, 2011. \MR{2908346}

\bibitem[Lan60]{MR0129065}
H.~J. Landau, \emph{On the recovery of a band-limited signal, after
  instantaneous companding and subsequent band limiting}, Bell System Tech. J.
  \textbf{39} (1960), 351--364. \MR{0129065}

\bibitem[Lan64]{MR0206615}
\bysame, \emph{A sparse regular sequence of exponentials closed on large sets},
  Bull. Amer. Math. Soc. \textbf{70} (1964), 566--569. \MR{0206615}

\bibitem[Lan67]{MR0222554}
\bysame, \emph{Necessary density conditions for sampling and interpolation of
  certain entire functions}, Acta Math. \textbf{117} (1967), 37--52.
  \MR{0222554}

\bibitem[LB04]{MR2089140}
Yi~Lin and Lawrence~D. Brown, \emph{Statistical properties of the method of
  regularization with periodic {G}aussian reproducing kernel}, Ann. Statist.
  \textbf{32} (2004), no.~4, 1723--1743. \MR{2089140 (2006a:62053)}

\bibitem[LLSB84]{doi:10.1137/0144089}
H.~J. Landau, B.~F. Logan, L.~A. Shepp, and N.~Bauman, \emph{Diffusion, cell
  mobility, and bandlimited functions}, SIAM Journal on Applied Mathematics
  \textbf{44} (1984), no.~6, 1232--1245.

\bibitem[LP61]{MR0140733}
H.~J. Landau and H.~O. Pollak, \emph{Prolate spheroidal wave functions,
  {F}ourier analysis and uncertainty. {II}}, Bell System Tech. J. \textbf{40}
  (1961), 65--84. \MR{0140733}

\bibitem[LP11]{MR2975345}
Sneh Lata and Vern Paulsen, \emph{The {F}eichtinger conjecture and reproducing
  kernel {H}ilbert spaces}, Indiana Univ. Math. J. \textbf{60} (2011), no.~4,
  1303--1317. \MR{2975345}

\bibitem[Nel57]{Nel57}
Edward Nelson, \emph{Kernel functions and eigenfunction expansions}, Duke Math.
  J. \textbf{25} (1957), 15--27. \MR{0091442 (19,969f)}

\bibitem[NSW11]{MR2810909}
P.~Niyogi, S.~Smale, and S.~Weinberger, \emph{A topological view of
  unsupervised learning from noisy data}, SIAM J. Comput. \textbf{40} (2011),
  no.~3, 646--663. \MR{2810909}

\bibitem[OST13]{MR3054607}
Kasso~A. Okoudjou, Robert~S. Strichartz, and Elizabeth~K. Tuley,
  \emph{Orthogonal polynomials on the {S}ierpinski gasket}, Constr. Approx.
  \textbf{37} (2013), no.~3, 311--340. \MR{3054607}

\bibitem[Pau02]{MR1976867}
Vern Paulsen, \emph{Completely bounded maps and operator algebras}, Cambridge
  Studies in Advanced Mathematics, vol.~78, Cambridge University Press,
  Cambridge, 2002. \MR{1976867}

\bibitem[RW06]{RW06}
Carl~Edward Rasmussen and Christopher K.~I. Williams, \emph{Gaussian processes
  for machine learning}, Adaptive Computation and Machine Learning, MIT Press,
  Cambridge, MA, 2006. \MR{2514435 (2010i:68131)}

\bibitem[Sei04]{MR2040080}
Kristian Seip, \emph{Interpolation and sampling in spaces of analytic
  functions}, University Lecture Series, vol.~33, American Mathematical
  Society, Providence, RI, 2004. \MR{2040080}

\bibitem[SRB{\etalchar{+}}10]{MR2591839}
S.~Smale, L.~Rosasco, J.~Bouvrie, A.~Caponnetto, and T.~Poggio,
  \emph{Mathematics of the neural response}, Found. Comput. Math. \textbf{10}
  (2010), no.~1, 67--91. \MR{2591839}

\bibitem[SS01]{SchlkopfSmola200112}
Bernhard Schlkopf and Alexander~J. Smola, \emph{Learning with kernels: Support
  vector machines, regularization, optimization, and beyond (adaptive
  computation and machine learning)}, 1st ed., The MIT Press, 12 2001.

\bibitem[SS13]{MR3101840}
Oded Schramm and Scott Sheffield, \emph{A contour line of the continuum
  {G}aussian free field}, Probab. Theory Related Fields \textbf{157} (2013),
  no.~1-2, 47--80. \MR{3101840}

\bibitem[ST12]{MR2892621}
Robert~S. Strichartz and Alexander Teplyaev, \emph{Spectral analysis on
  infinite {S}ierpi\'nski fractafolds}, J. Anal. Math. \textbf{116} (2012),
  255--297. \MR{2892621}

\bibitem[STC04]{Shawe-TaylorCristianini200406}
John Shawe-Taylor and Nello Cristianini, \emph{Kernel methods for pattern
  analysis}, Cambridge University Press, 2004.

\bibitem[Str10]{MR2764237}
Robert~S. Strichartz, \emph{Transformation of spectra of graph {L}aplacians},
  Rocky Mountain J. Math. \textbf{40} (2010), no.~6, 2037--2062. \MR{2764237
  (2012c:05199)}

\bibitem[SY06]{MR2228737}
Steve Smale and Yuan Yao, \emph{Online learning algorithms}, Found. Comput.
  Math. \textbf{6} (2006), no.~2, 145--170. \MR{2228737}

\bibitem[SZ05]{MR2186447}
Steve Smale and Ding-Xuan Zhou, \emph{Shannon sampling. {II}. {C}onnections to
  learning theory}, Appl. Comput. Harmon. Anal. \textbf{19} (2005), no.~3,
  285--302. \MR{2186447}

\bibitem[SZ09]{MR2488871}
\bysame, \emph{Online learning with {M}arkov sampling}, Anal. Appl. (Singap.)
  \textbf{7} (2009), no.~1, 87--113. \MR{2488871}

\bibitem[Vul13]{MR3091062}
Mirjana Vuleti{\'c}, \emph{The {G}aussian free field and strict plane
  partitions}, 25th {I}nternational {C}onference on {F}ormal {P}ower {S}eries
  and {A}lgebraic {C}ombinatorics ({FPSAC} 2013), Discrete Math. Theor. Comput.
  Sci. Proc., AS, Assoc. Discrete Math. Theor. Comput. Sci., Nancy, 2013,
  pp.~1041--1052. \MR{3091062}

\bibitem[ZXZ12]{MR2913695}
Haizhang Zhang, Yuesheng Xu, and Qinghui Zhang, \emph{Refinement of
  operator-valued reproducing kernels}, J. Mach. Learn. Res. \textbf{13}
  (2012), 91--136. \MR{2913695}

\end{thebibliography}

\end{document}